\documentclass[11pt]{article}

\usepackage{amsmath}
\usepackage{epsfig, amssymb, amsfonts}
\usepackage{amsthm}
\usepackage{amstext}
\usepackage{amsopn}
\usepackage{mathrsfs}
\usepackage{subfigure}
\usepackage{graphicx}
\usepackage[left=2.3cm,top=2cm,right=2.3cm,bottom=2cm, nohead,foot=1cm]{geometry}
\numberwithin{equation}{section}

\newtheorem{theorem}{Theorem}[section]
\newtheorem{identities}[theorem]{Identities}
\newtheorem{lemma}[theorem]{Lemma}

\newtheorem{remark}[theorem]{Remark}
\newtheorem{proposition}[theorem]{Proposition}

\newcommand{\R}{{\mathbb R}}

\newcommand{\C}{{\mathbb C}}

\title{Optimizing a variable-rate diffusion  to hit an infinitesimal target  at a  set time}

\author{\textbf{Jeremy Thane Clark}\footnote{jtclark@math.msu.edu} \\  Department of Mathematics, Michigan State University  \\ East Lansing, MI 48824, USA }

\begin{document}
\maketitle

\begin{abstract}

I consider a stochastic optimization problem for a  time-changed Bessel process whose diffusion rate is constrained to be between two  positive values $r_{1}<r_{2}$.  The problem is to find an optimal adapted strategy for the choice of diffusion rate in order to maximize the chance of hitting an infinitesimal region around the origin at a set time in the future.  More precisely, the parameter associated with ``the chance of hitting the origin" is the exponent for a singularity induced at the origin of the final time probability density.  I show that the optimal exponent solves a transcendental equation depending on the ratio $\frac{r_{2}}{r_{1}}$ and the dimension of the Bessel process.

\end{abstract}


\section{Introduction}

Pick $a\in \R^n$ and positive numbers $r_{1}$, $r_{2}$, $T$ with $r_{1}< r_{2}$.      For a Borel measurable function $D:\R^+\times [0,T]\rightarrow [r_{1},r_{2}]$,  let   $X_{t}\in \R^n$ be the  weak solution  to  the  stochastic differential equation
\begin{align}\label{Diffuse}
  dX_{t}^{j}= \sqrt{D\big(|X_{t}|,t\big)}d\mathbf{B}_{t}^{j}, \hspace{1.5cm} X_{0}=a, \quad 1\leq j\leq n, \quad  t\in [0,T] , \end{align}
for a standard $n$-dimensional Brownian motion $\mathbf{B}_{t}$. 
 In broad terms the question I address in this article is the following: What choice of diffusion coefficient  maximizes the probability that  $X_{t}$ lands in an infinitesimal region around the origin at the final time $T$ given the constraint $r_{1}\leq D(x,t)\leq r_{2}$?    
Before stating the problem more precisely, I will switch to a framework that allows for fractional dimensions $n\in \R^{+}$.  
Since the setup above is spherically symmetric, it is natural to postulate the problem in terms of the time-changed, n-dimensional Bessel process $\mathbf{x}_{t}:=|X_{t}|\in \R^{+}$, which has transition densities $\mathcal{P}^{(D)}_{y,t}\in L^{1}(\R^{+})$  satisfying the forward Kolmogorov equation
\begin{align}\label{EqBessel}
\frac{d}{dt}\mathcal{P}^{(D)}_{y,t}(x)=-\frac{1}{2}\frac{d}{dx}\Big(\frac{n-1}{x}D(x,t)\mathcal{P}^{(D)}_{y,t}(x)    \Big) +\frac{1}{2}\frac{d^2}{dx^{2}}\Big(D(x,t)\mathcal{P}^{(D)}_{y,t}(x) \Big), \quad \mathcal{P}^{(D)}_{y,0}(x):=\delta_{y}(x),
\end{align}
for $x\in \R^+$, $y=|a|$, and $t\in [0,T]$.   If the diffusion coefficient is constant, i.e.,  $D(x,t)=r>0$ for all $(x,t)$, then $\mathbf{x}_{t}$ is a Bessel process and  the behavior of the final time density for  $x\ll 1$ will be $\mathcal{P}^{(D)}_{y,T}(x)\sim x^{\epsilon}$, where $\epsilon=n-1$; see~\cite{Revuz} for an explicit expression for the transition semigroup of  a Bessel process.

   It turns out that maximizing the chance of landing in ``an infinitesimal region around the origin" at time $T$ does not merely mean maximizing the coefficient in front of the asymptotic power $x^{\epsilon}$,  because smaller values of $\epsilon$ can be attained through better choices of $D(x,t)$.    Thus the problem shifts to minimizing the exponent of the asymptotic power law  for  $\mathcal{P}_{y,T}^{(D)}(x)$ at the origin, which I will  characterize through the  limit
\begin{align}\label{JizJot}
\hspace{3cm}\mathcal{I}(D)=\liminf_{\epsilon\searrow 0}\left\{n-\frac{\log\Big(\int_{[0, \epsilon]}dx \mathcal{P}^{(D)}_{y,T}(x)  \Big)}{  \log(\epsilon) }\right\}\in (-\infty,n).   
\end{align}  
This definition is designed so that $I(D)$ is the improvement of the asymptotic power  over the case in which  $D(x,t)$ is constant:   If $ \mathcal{P}^{(D)}_{y,T}(x) \sim x^{n-1-\eta}$ around $x=0$ for $\eta>0$, then $\mathcal{I}(D)=\eta$, and, in  
particular, $\mathcal{I}(D)=0$ when $D(x,t)$ is constant by the observation above.  My focus in this article will be primarily on dimensions $n\in (0,2)$ since some of the formulas that I use blow up for $n\geq 2$.   It is not surprising that a transition in behavior should occur around dimension $n=2$, where Bessel processes transition from recurrent to transient.

If we think of $D(x,t)$ as the strategy of a random walker $\mathbf{x}_{t}$ attempting to maximize his chance of arriving at the origin at time $T$, it is reasonable that he should rush with the maximum diffusion rate $r_{2}$ when he judges himself to be  far given the time remaining, and he should  choose to bide his time with the minimum diffusion rate $r_{1}$ when he judges himself to be close.   Thus it is natural to have  $D(x,t)\nearrow r_{2}$ as $x\nearrow \infty$ for each $t\in [0,T)$.    
   Since  $\mathbf{x}_{t}$ is a time-changed Bessel process with diffusion rates restricted to the interval $[r_{1},r_{2}]$,  my optimization problem inherits a  scale-invariance when viewed from the origin and the final time $T$; the random walker should make the same choice of diffusion rate at space-time points $(x,t)$ and $(x',t')$  in $\R^{+}\times [0,T)$ for which
$  \frac{x^{2} }{T-t} = \frac{x'^{2}}{T-t'} $.
Any strategy $ D(x,t)$ consistent with the above scale-invariance satisfies  
\begin{align}\label{DiffCoef}
 D(x,t)=D\Big(x\sqrt{\frac{T-t'}{T-t} }  , t'\Big),  \hspace{2cm} t,t'\in [0,T).  
\end{align}  
For the above reasons, I will focus my analysis on  diffusion coefficients of the form $D(x,t)=R\big(\frac{x}{\sqrt{T-t}}\big)$ for measurable functions $R:\R^{+}\rightarrow [r_{1},r_{2}]$ with $\lim_{z\rightarrow \infty}R(z)=r_{2}$.  I denote the set of such $R$ by $ B_{r_{1},r_{2}}$.

Theorem~\ref{ThmMain} is the main result of this article.   To state the result we need to define positive numbers $\kappa$ and $\eta$ solving the pair of equations~(\ref{Krakow}), which depend on the dimension $n\in (0,2)$ and the diffusion bounds $r_{1}$, $r_{2}$ through their ratio $V:=\sqrt{\frac{r_{2}}{r_{1}}}$.  For $\eta>0$ and $\nu>-1$ define 
$ S_{ \nu  }^{-}(x) :=x^{-\nu  }I_{\nu}(x) $  and $ S_{ \nu  }^{+}(x) :=x^{-\nu  }K_{\nu}(x)$, 
where $I_{\nu}, K_{\nu}:\R^{+}\rightarrow \R^{+}$ are modified Bessel functions of the first and second kind, respectively. 
Define the functions $Y_{\nu, \eta}^{\pm}:\R^+\rightarrow \R^{+}$ by 
$Y_{\nu, \eta }^{\pm}(x):= \int_{0}^{\infty}dz z^{\eta-1}S_{ \nu  }^{\pm}(xz)e^{-\frac{z^2+x^2}{2}}.  $
Given $V\in \R^{+}$ let the constants $\eta \equiv \eta(n,V)$ and  $\kappa\equiv \kappa(n,V)$ be determined by 
\begin{align}\label{Krakow}
\frac{Y_{\frac{n}{2},\eta+2 }^+(\frac{\kappa}{V})}{ Y_{\frac{n-2}{2},\eta }^+(\frac{\kappa}{V})  } =\frac{n-\eta-\frac{\kappa^{2}}{V^{2}}    }{ \frac{ \kappa^2  }{V^2}}  \hspace{1cm} \text{and} \hspace{1cm}
\frac{\kappa^{2-n}\int_{0}^{\kappa}da \, Y_{\frac{n-2}{2},\eta }^-(a)\, a^{n-1} e^{\frac{a^{2}-\kappa^2}{2}} }{  Y_{\frac{n-2}{2},\eta }^-(\kappa)   }= 1.
\end{align}
Note that because  $ K_{\nu}(x)\sim x^{-|\nu|}$ for  $0<x\ll 1$ the  integrals defining $Y_{\frac{n}{2},\eta+2 }^+$ and  $Y_{\frac{n-2}{2},\eta }^+$ blow up around zero when $n\geq 2$.

\begin{theorem}\label{ThmMain}
Fix $y\in \R^{+}$ and positive numbers $T$, $r_{1}$, $r_{2}$  with $r_{2}> r_{1}$.   Let $\mathcal{P}^{(R)}_{y,t}\in L^{1}(\R^+)$ obey the Kolmogorov equation~(\ref{EqBessel}) with $D(x,t)=R\big(\frac{x}{\sqrt{T-t}}\big)$.  For $V:=\sqrt{\frac{r_{2}}{r_{1}}}$ the following equality holds:
\begin{align}\label{Eta}
\eta(n, V)=\max_{ R\in B_{r_{1},r_{2}  } } \lim_{\epsilon\searrow 0}\left\{ n-\frac{\log\Big(\int_{[0. \epsilon]}dx \mathcal{P}^{(R)}_{y,T}(x)  \Big)}{  \log(\epsilon) }\right\}.
 \end{align} 
The above maximum is attained uniquely for $R^{*}:\R^{+}  \rightarrow [r_{1},r_{2}]$ of the form
\begin{align}\label{OptimalAr}
 R^{*}(x):=  r_{1}\chi\Big( \frac{ x}{\sqrt{r_{1}}}   \leq   \kappa (n,V)  \Big)+r_{2}\chi\Big(  \frac{x}{\sqrt{r_{1}}}   >   \kappa (n,V)  \Big) .    
\end{align}

\end{theorem}

\begin{remark}
It is instructive to examine the limiting behavior of the exponent $\eta(n,V)$ and the cut-off parameter $\kappa(n,V)$ characterizing the optimal solution in the respective limits $V \searrow 1$ and $V\nearrow \infty$.  
 One surprise is that for large $V$ the optimal cut-off  $\kappa(n,V)$  approaches a finite value $\overline{\kappa}_{n}\in \R^{+}$ solving  the equation 
$$  1=\overline{\kappa}_{n}^{2-n}\int_{0}^{\overline{\kappa}_{n}}da a^{n-1} e^{\frac{a^2-\overline{\kappa}_{n}^2}{2}}. $$  
Note that $\overline{\kappa}_{n}$ increases over the interval $n\in (0,2)$ and has the limiting behavior
$$ \frac{  \overline{\kappa}_{n} }{\sqrt{n}}\longrightarrow 1 \quad  \text{as} \quad n\searrow 0\hspace{1cm}\text{and} \hspace{1cm}\frac{\overline{\kappa}_{n} }{\sqrt{2\ln\big( \frac{1}{2-n}  \big)  } }\longrightarrow 1 \quad \text{as}\quad n\nearrow 2.  $$

The values $\eta(n,V)$, $\kappa(n,V)$ have the following characteristics for each $n\in (0,2)$.    
\begin{enumerate}

\item  $\eta(n,V)$, $\kappa(n,V)$ increase continuously with the parameter $V\in (1,\infty)$.  

\item  As $V\searrow 1$,  
$$\eta(n,V)\searrow 0 \hspace{1cm}\text{and} \hspace{1cm} \kappa(n,V)\searrow \sqrt{n} .$$ 

\item As $V\nearrow \infty$, 
$$ \eta(n,V)\nearrow n  \hspace{1cm} \text{and}\hspace{1cm} \kappa(n,V)\nearrow \overline{\kappa}_{n}  . $$

\item Moreover, for large $V$,
$$   n-\eta(n,V)\propto V^{n-2}\hspace{1cm}\text{and}\hspace{1cm}  \overline{\kappa}_{n}-\kappa(n,V)\propto V^{n-2}.  $$

\end{enumerate}

By item 3 the exponent $\eta(n,V)$ approaches its upper limit $n$ as the ratio $V=\frac{\sqrt{r_{2}}}{ \sqrt{r_{1}} }$  goes to infinity, however, item 4 illustrates that this convergence occurs more slowly as $n$ gets closer to $2$.    This is not surprising since, intuitively, the Bessel process becomes more weakly recurrent as $n$ approaches $2$, and a  higher value of $r_{2}$ in comparison to $r_{1}$  is needed to speed up the returns of the random walk to the region around the origin.

\end{remark}

\subsection{Further Discussion}

Borkar~\cite{Borkar} and Fleming~\cite{Fleming} are reference books for optimization in stochastic settings. 
The optimization problem described above focuses on maximizing the probability of certain vanishingly low chance events.  In particular there is no penalty  for landing far from the target region.  It is a much different problem, for instance, to minimize a quantity of the form
\begin{align}\label{Hamlet}
\widetilde{\mathcal{I}}_{y,T}(D)=  \int_{\R^+}dx  \mathcal{P}_{y,T}^{(D)}(x) \varphi(x),
\end{align}
where  $\mathcal{P}_{y,T}^{(D)}$ is defined as in~(\ref{EqBessel}) and $\varphi:\R^+\rightarrow \R^{+}$ is a convex function quantifying the penalty for landing away from the target point   at the final time $T$.   When $D(x,t)$ is restricted to the range $[r_{1},r_{2}]$, the optimal strategy for the penalty problem is simply to always use the lowest available diffusion rate $r_{1}$.   If the goal is to maximize~(\ref{Hamlet}) for a given target function $\varphi:\R^+\rightarrow \R^{+}$, e.g., $\varphi(x)=1_{[0, 1]}(x)$, then the maximizing strategy can be formally derived from the solution of a nonlinear differential equation;  the optimal, maximizing strategy $D^{*}(x,t)$ should  have the form
\begin{align}\label{Dingbell}
  D^{*}(x,t)=  r_{1}\chi\Big(    \big(\Delta_{n}G\big)(x,T-t) \leq 0  \Big)    +  r_{2}\chi\Big(     \big(\Delta_{n}G\big)(x,T-t) > 0  \Big) ,
\end{align}
where $\Delta_{n}:=\frac{n-1}{x}\frac{d }{d x}+\frac{d^2}{d x^2}$ is the $n$-dimensional spherical Laplacian and $G:\R^{+}\times [0,T]\rightarrow \R^{+}$ is the solution to the nonlinear backwards Kolmogorov equation
$$      \frac{d}{dt}G(x,t)=   \frac{1}{2} \Big[ r_{1}\chi\Big(    \big(\Delta_{n}G\big)(x,t) \leq 0  \Big)    +  r_{2}\chi\Big(    \big(\Delta_{n}G\big)(x,t) > 0  \Big) \Big]\big(\Delta_{n}G\big)(x,t)   $$
 with initial condition  $G(x,0)= \varphi(x) $.   The form of the maximizing strategy~(\ref{Dingbell}) is a consistency requirement since 
any strategy not satisfying~(\ref{Dingbell}) will admit a locally perturbated strategy $D'=D^{*}+dD^{*}$ yielding a small improvement in the value $\widetilde{\mathcal{I}}_{y,T}(D')$.

My interest, however, is in the largest possible exponent with which~(\ref{Hamlet}) decays as the target function shrinks, i.e., $\varphi^{(\epsilon)}(x):=1_{[0, \epsilon]}(x)$ for $0<\epsilon\ll 1 $.   Through a space-time transformation and some analysis, my optimization problem amounts to finding the $R\in B_{r_{1},r_{2}}$\footnote{Recall that $B_{r_{1},r_{2}}$ is defined as the space of measurable functions from $\R^{+}$ to $[r_{1},r_{2}]$ satisfying $\lim_{z\rightarrow \infty}R(z)=r_{2}$.  } that maximizes the principle eigenvalue for differential  operators $\mathcal{L}^{(R)}$ defined over certain weighted $L^{2}$-spaces and having the form
\begin{align}\label{Olf}
  \mathcal{L}^{(R)}:=x\frac{d}{dx}+R(x)\Delta_{n} .
\end{align}
The maximized principle eigenvalue is $\lambda_{r_{1},r_{2}}^{(n)}:=\eta(n,V)-n$ and the corresponding eigenvector is
\begin{align*}
\phi_{r_{1},r_{2}}^{(n)}(x):= \gamma Y_{\frac{n-2}{2},\eta(n ,V) }^{-}\left( \frac{x}{\sqrt{r_{1}} }\right)\chi\left(  \frac{x }{ \sqrt{r_{1}} }\leq \kappa(n,V)   \right) +  Y_{\frac{n-2}{2}, \eta(n,V) }^+\left(\frac{x}{\sqrt{r_{2}} }\right)\chi\left(  \frac{x}{\sqrt{r_{1}}    } >  \kappa(n,V)\right) 
\end{align*}
for $\gamma>0$ chosen to make the function continuous at $x=\sqrt{r_{1}}\kappa(n,V)$.  The maximized principle eigenvalue also serves as the principle eigenvalue for a fully nonlinear elliptic operator:
\begin{align}\label{Dis}
  F_{r_{1},r_{2}}\left(x, \frac{d \phi_{r_{1},r_{2}}^{(n)} }{d x}   , \Delta_{n} \phi_{r_{1},r_{2}}^{(n)} \right)= \lambda_{r_{1},r_{2}}^{(n)}\phi_{r_{1},r_{2}}^{(n)} ,
\end{align} 
where $F_{r_{1},r_{2}}:\R^{+}\times \R^{2}$ decays at infinity and 
\begin{align*}
\hspace{3cm}F_{r_{1},r_{2}}(x,y,z):= \frac{xy}{2}+ \frac{z}{2}\Big[ r_{1}\chi(z<0)+ r_{2}\chi(z\geq 0)\Big]. 
\end{align*}
A basic discussion of  principle eigenvalues for linear  elliptic operators can be found in Pinsky's book~\cite{Pinsky}.  Theory on principle eigenvalues for fully nonlinear elliptic operators is developed  in~\cite{Pucci,Busca,Quaas,Birindelli,Birindelli2,Bayraktar}.   In particular, the eigenvalue problems studied in~\cite{Pucci,Busca, Bayraktar} are similar in character to~(\ref{Dis}) except with  $\Delta_{n}$ replaced by the second derivative (and generalized to arbitrary dimension). 
The theory in~\cite{Bayraktar} is applied to a problem suggested
in~\cite{Kardaras} regarding robust asymptotic growth rates for financial derivatives with unknown  underlying volatility rates.   \vspace{.25cm}

Note that if the problem is to maximize the expected amount of time that the random walker spends in the interval $[0,\epsilon]$ up to time $T$, then the problem becomes trivial and improved exponents can not be attained by using variable diffusion coefficients,  $D(x,t)\in [r_{1},r_{2}]$.   The optimal strategy is obviously to linger when in $[0,\epsilon]$ and hurry when in $(\epsilon,\infty)$:   $D(x,t)=r_{1}\chi(x\leq \epsilon)+r_{2}\chi(x>\epsilon)$.   Moreover, by thinking of $\mathbf{x}_{t}$ as a stochastic time-change of the Bessel process $\mathbf{\widehat{x}}_{t}$ with $D(x,t)=r_{1}$,  I am lead to  the bound   
\begin{align}
\mathbb{E}_{y}\bigg[ \int_{0}^{T}dt\chi\big(\mathbf{x}_{t}\leq \epsilon      \big)     \bigg]\leq  \mathbb{E}_{y}\bigg[ \int_{0}^{ T\frac{r_{2}}{r_{1}} }dt\chi\big(\mathbf{\widehat{x}}_{t}\leq \epsilon      \big)     \bigg]\propto \epsilon^{n} , \hspace{1cm} \epsilon \ll 1,
\end{align}
  Thus the shrinking target zone is not interesting for this problem.

The remainder of this article is organized as follows:
\begin{itemize}

\item In Sect. 2  I introduce the simple space-time transformation that links the original time-changed Bessel process to a stationary dynamics generated by operators of the form~(\ref{Olf}).  Except for the proof of Thm.~\ref{ThmMain}, all of the remaining parts of this article concern results for $\mathcal{L}^{(R)}$.

\item Section 3 establishes the self-adjointness of $\mathcal{L}^{(R)}$ in a weighted Hilbert space and derives some general results for the principle eigenvalue and its corresponding eigenfunction.

\item In Sect. 4 I show that the problem of maximizing the principle eigenvalue for operators $\mathcal{L}^{(R)}$ can be restricted to the class of $R:\mathbb{R}^{+}\rightarrow [r_{1},r_{2}]$ of the form $R(x)=r_{1}\chi(x\leq \frak{c})+r_{2}\chi(x>\frak{c})$ for some $\frak{c}>0$.  In terms of the random walker, this implies that an optimizing strategy should always switch between the extremal diffusion rates $r_{1}$ and $r_{2}$.    I also derive that the maximal possible principle eigenvalue is $\eta(n,V)-n$ and occurs when the cut-off is $\frak{c}=r_{1}\kappa(n,V)$.

\item Section 5 contains the proof of Thm. 1.1.

\end{itemize}

\section{The stationary dynamics}

The restriction of the diffusion coefficient $D(x,t)$ to the parabolic form $R\big(\frac{x}{\sqrt{T-t}}\big)$ for a measurable function $R:\R^{+}\rightarrow [r_{1},r_{2}]$ implies that a solution to the Kolmogorov equation~(\ref{EqBessel}) is equivalent under a time-space reparameterization to the solution of a stationary dynamics~(\ref{DiffEq}).  For $(x,t)\in \R^{+}\times [0,T)$ let $(z,s)\in \R^{+}\times \R^{+}$ be given by
\begin{align}\label{Trans}
(x,t) \quad \longrightarrow \quad (z,s)=\Big( \frac{x}{\sqrt{T-t}},  \log\Big(\frac{T}{T-t} \Big)   \Big).
\end{align}
Through the transformation~(\ref{Trans}), we can use $\mathcal{P}^{(R)}_{y,T}(x)$ to define new probability densities $\psi_{b,s}^{(R)}(z):=  \sqrt{T} e^{-\frac{1}{2}s} \mathcal{P}^{(R)}_{\sqrt{T}b,\, T-Te^{-s}}\big(\sqrt{T}e^{-\frac{1}{2}s}z\big)$ satisfying the forward equation
\begin{align}\label{DiffEq}
\frac{d}{ds}\psi_{b,s}^{(R)}(z)=-\frac{1}{2}\frac{d}{dz}\Big( z  \psi_{b,s}^{(R)}(z)\Big)-\frac{1}{2}\frac{d}{dz}\Big(R(z)\frac{n-1}{z}   \psi_{b,s}^{(R)}(z)\Big)+\frac{1}{2}\frac{d^2}{dz^2}\Big(R(z) \psi_{b,s}^{(R)}(z)  \Big),
\end{align}
where $b:=\frac{y}{\sqrt{T}}$, $s\in [0,\infty)$, and $\psi_{b,0 }^{(R)}(z)=\delta_{b}(z)$.  The backward Kolmogorov generator is thus $\frac{1}{2}\mathcal{L}^{(R)}$ for $\mathcal{L}^{(R)}:=x\frac{d}{dx}+R(x)\Delta_{n}$.  The diffusion process $Z_{t}$ corresponding to~(\ref{DiffEq}) has a repulsive drift that grows proportionatly to the distance from the origin:
\begin{align}\label{D}
dZ_{s}=\frac{Z_{s}}{2}ds+R(Z_{s})\frac{n-1}{2Z_{s}}ds+\sqrt{R(Z_{s})}d\mathbf{B}_{s}',\hspace{1cm}Z_{0}=b,
\end{align}
where $\mathbf{B}_{s}'$ is a copy of standard Brownian motion.  When $R(z)$ is a constant function, $Z_{s}$ is an $n$-dimensional radial Ornstein-Uhlenbeck process; see~\cite{Anja} or~\cite{Revuz} for discussion of radial Ornstein-Uhlenbeck processes.   In the next section, I will show that the generator  $\frac{1}{2}\mathcal{L}^{(R)}$ is self-adjoint when assigned the appropriate domain, which guarantees the existence  of the dynamics.

The trajectories for the processes $Z_{s}$ will undergo an essentially exponential divergence to infinity after wandering near the origin for a finite time period.  The state of the original process $\mathbf{x}_{t}$ at the final time $T$ is recovered by the limit
\begin{align}\label{LimitEX}
\mathbf{x}_{T}=\lim_{s\rightarrow \infty}\frac{\sqrt{T}Z_{s}}{e^{\frac{s}{2}}}=Z_{0}+\int_{0 }^{\infty}e^{-\frac{s}{2}}\Big(\frac{n-1}{2Z_{s}}ds+\sqrt{R(Z_{s})}d\mathbf{B}_{s}   \Big).
\end{align}

\section{Analysis of the generators for the  stationary dynamics}\label{SecGenerators}

Let $B(\R^{+},[r_{1},r_{2}]    )$ denote the collection of Borel measurable functions from $\R^{+}$ to $[r_{1},r_{2}]$.  As mentioned in the last section, for a given element $R\in B(\R^+,[r_{1},r_{2}])$, the backwards generator for the stationary dynamics has the form $\frac{1}{2}\mathcal{L}^{(R)}$ for $\mathcal{L}^{(R)}:=x\frac{d}{dx}+R(x)\Delta_{n}$, where $\Delta_{n}$ is the radial Laplacian, $\Delta_{n}:=\frac{n-1}{x}\frac{d}{dx}+\frac{d^2}{dx^2}$.  The next lemma states that the operator $\mathcal{L}^{(R)}$ is self-adjoint when acting on the weighted $L^{2}$-space defined below.  Let  $L^{2}\big(\R^{+},w(x)dx\big)$ be the Hilbert space with inner product
$$ \langle f | g\rangle_{R}:=\int_{\R^{+}}dx w(x)\overline{f(x)}g(x)\quad \text{for weight}\quad w(x):=\frac{x^{n-1}e^{\int_{0}^{x}dv\frac{v}{R(v)}} }{ R(x)  }. $$
The corresponding norm is denoted by $\|f\|_{2,R}:=\sqrt{\langle f | f\rangle_{R}}$.

\begin{proposition}\label{LemSelfAdj}
 Let $R\in B(\R^{+},[r_{1},r_{2}] )$.   The operator $\mathcal{L}^{(R)}$ is self-adjoint when assigned the domain  $$\mathbf{D}=\left\{f\in L^{2}\big(\R^{+},w(x)dx\big)\, \bigg|\, \big\|\Delta_{n}f\big\|_{2,R}  <\infty \quad \text{and}\quad  \lim_{x\searrow 0}x^{n-1}\frac{df}{dx}(x)=0          \right\} . $$
Moreover,  $(\mathcal{L}^{(R)},\mathbf{D})$ and  $( \Delta_{n},\mathbf{D})$ are mutually relatively bounded.  \vspace{.3cm}

\end{proposition}

Before going to the proof of Prop.~\ref{LemSelfAdj}, I will prove the following simple lemma. 
\begin{lemma}[Closure Property]\label{Nill}
The space $ \mathbf{D}$ is closed with respect to the graph  norm $\|  g\|_{\Delta_{n}}:=\|g\|_{2,R}+\|\Delta_{n} g\|_{2,R}$.
\end{lemma}
\begin{proof}
  Let $f_{j}$ be a Cauchy sequence with respect to the norm $\|  \cdot \|_{\Delta_{n}}$.   There are $f,g\in L^{2}\big(\mathbb{R}^{+},w(x)dx\big)$ such that  
\begin{align}\label{Mig}
 \|f_{j}-f   \|_{2,R}\longrightarrow 0 \hspace{1cm}\text{and}\hspace{1cm} \| \Delta_{n}f_{j}-g   \|_{2,R}\longrightarrow 0      .
 \end{align}
To show $f\in \mathbf{D}$, I need to verify that $g=\Delta_{n}f$ and $\lim_{x\searrow 0}x^{n-1}\frac{df}{dx}(x)=0$.

 It will be useful to use a spatial transformation.  Define $\widehat{h}(z):=h\big( z^{\frac{1}{2-n} } \big)$ for arbitrary $h:\R^+\rightarrow \C$.  Notice that  $(\Delta_{n}h)(x) =(2-n)^2 z^{\frac{2(1-n) }{2-n  } }\frac{d^2\widehat{h}}{dz^2}(z)   $ for $z=x^{2-n}$.   The equality $g=\Delta_{n}f$ is equivalent to $\widehat{g}=(2-n)^2 z^{\frac{2(1-n) }{2-n  } }\frac{d^2\widehat{f}}{dz^2} $ and, by calculus,  this is equivalent to $\widehat{f}=h$ for 
\begin{align*}
h(z):= \int_{z}^{\infty}da (a-z) \frac{a^{\frac{2(n-1) }{2-n  } }}{(2-n)^2}\widehat{g}(a).
\end{align*}
To see that $\widehat{f}=h$ indeed  holds, notice that
\begin{align*}
|\widehat{f}(z)-h(z)|  \leq &      \big|\widehat{f}(z)-\widehat{f}_{j}(z)\big| +   \big|\widehat{f}_{j}(z)-h(z)\big| \\   \leq & \big|\widehat{f}(z)-\widehat{f}_{j}(z)\big| +\bigg| \int_{z}^{\infty}da (a-z)\bigg( \frac{d^2\widehat{f}_{j}}{dz^2}(a) - \frac{ a^{\frac{2(n-1) }{2-n  } } }{(2-n)^2}\widehat{g}(a) \bigg) \bigg|  \\   \leq &  \big|\widehat{f}(z)-\widehat{f}_{j}(z)\big| + \frac{ \| \Delta_{n} f_{j} -g \|_{2,R}   }{(2-n)^{\frac{3}{2} } }\Bigg(\int_{z}^{\infty}da (a-z)^2 \frac{a^{\frac{3(n-1) }{ 2-n } }   }{w\big(a^{\frac{1}{2-n}} \big)}    \Bigg)^{\frac{1}{2}},
\end{align*}
where the third inequality follows by Cauchy-Schwarz and a change of integration variables.  Moreover,~(\ref{Mig}) implies 
that for  a.e. $z\in \R^+$  there is a subsequential limit  $j_{m}\rightarrow \infty$ such that the right side above converges to zero. 
Thus $g=\Delta_{n}f$.

I can use similar techniques to show that $\lim_{x\searrow 0}x^{n-1}\frac{df}{dx}(x)=0$.  Notice that  $\lim_{x\searrow 0}x^{n-1}\frac{df}{dx}(x)=0$ is equivalent to $\lim_{z\searrow 0}\frac{d\widehat{f}}{dz}(z)=0$.   Define $h_{j}:=f_{j}-f$.  By  calculus, I have
\begin{align*}
\frac{d\widehat{h}_{j}}{dz}(z)=-\int_{z}^{\infty}da \frac{d^2\widehat{h}_{j}}{dz^2}(a).
\end{align*}
Cauchy-Schwarz and changes of integration variables yield
\begin{align}
\Big|\frac{d\widehat{h}_{j}}{dz}(z)\Big|\leq  \sqrt{C} \|     h_{j}   \| _{\Delta_{n}} \hspace{1cm}\text{for}\hspace{1cm}C:=  \int_{0}^{\infty}dx x^{\frac{2(n-1)}{2-n}   }R(x)  e^{-\int_{0}^{x}da\frac{a}{R(a)} }  .   
\end{align}
It follows that $\frac{d\widehat{h}_{j}}{dz}=\frac{d\widehat{f}_{j}}{dz}-\frac{d\widehat{f}}{dz}$ converges to zero uniformly as $j\rightarrow \infty$, and I have  $\lim_{x\searrow 0}\frac{d\widehat{f} }{dx}(x)=0$.\vspace{.3cm}

\end{proof}

\begin{proof}[Proof of Proposition~\ref{LemSelfAdj}]
In the analysis below, I will prove the following technical points:
\begin{enumerate}
\item[(i).]  $\mathcal{L}^{(R)}$ sends elements in $\mathbf{D}$ to $L^{2}\big(\mathbb{R}^{+},w(x)dx\big)$, i.e.,  $\mathcal{L}^{(R)}$ is well-defined on the space $\mathbf{D}$.

\item[(ii).]  For all $f \in \mathbf{D}$,   there is a $C>0$ such that 
$$  \big\|\mathcal{L}^{(R)}f\big\|_{2,R}   \leq  C\big( \|  f \|_{2,R}   +\big\|    \Delta_{n} f\big\|_{2,R}\big).  $$

\item[(iii).]  For all $f \in \mathbf{D}$,  there is a $C>0$ such that 
$$  \big\|    \Delta_{n} f\big\|_{2,R}  \leq  C\big( \|  f \|_{2,R}   +  \big\|\mathcal{L}^{(R)}f\big\|_{2,R} \big).  $$

\end{enumerate}
Before proving the above statements, I will use them to deduce that  $(\mathcal{L}^{(R)},\mathbf{D})$ is self-adjoint.  It is sufficient to show that $(\mathcal{L}^{(R)},\mathbf{D})$ is symmetric and has no nontrivial  extension (since the adjoint of a symmetric operator is a closed extension).   Two applications of integration by parts shows that $(\mathcal{L}^{(R)},\mathbf{D})$ is a symmetric operator since for all $f,g\in \mathbf{D} $
\begin{align*}
 \big\langle g \big| \mathcal{L}^{(R)}f    \big\rangle_{R}=  \int_{\R^{+}}dx x^{n-1}e^{\int_{0}^{x}dv\frac{v}{R(v)}}\overline{\frac{dg}{dx}(x)}\frac{df}{dx}(x)=  \big\langle \mathcal{L}^{(R)} g \big| f    \big\rangle_{R} ,
\end{align*}
where the boundary terms vanish by the condition $\lim_{x\searrow 0}x^{n-1}\frac{d\phi}{dx}(x)=0$ for $\phi=f,g$.  Suppose that there are $f_{j}\in \mathbf{D}$ and  $f,g\in L^{2}\big(\R^{+},w(x)dx\big)$ such that
$$   \| f_{j}-f\|_{2,R}\longrightarrow 0  \hspace{1cm}\text{and}\hspace{1cm}\big\| \mathcal{L}^{(R)}f_{j}-g\big\|_{2,R} \longrightarrow 0 \hspace{1cm}\text{as}\hspace{1cm}j\longrightarrow 0 .    $$
Since $f_{j}$ and $\mathcal{L}^{(R)}f_{j}$  are Cauchy in $L^{2}\big(\R^{+},w(x)dx\big)$, statement (iii) implies that $\Delta_{n}f_{j}$ is also Cauchy.  By Lem.~\ref{Nill}, it follows that $f$ is in $D$.  Thus, $(\mathcal{L}^{(R)},\mathbf{D})$ has no nontrivial extension and  must be self-adjoint. \vspace{.2cm}

 To complete the proof, I  will now prove statements (i)-(iii).  \vspace{.4cm}

\noindent (i) and (ii).   Using integration  by parts, I have the equality below for all smooth functions   $f\in L^{2}\big(\R^{+},w(x)dx\big)$ with $\Delta_{n}f\in L^{2}\big(\R^{+},w(x)dx\big)$:
\begin{align}\label{IntPart}
 \big\|\mathcal{L}^{(R)}f\big\|_{2,R}^{2}= \big\|R(x)\Delta_{n} f  \big\|_{2,R}^{2}- \int_{\R^{+}}dx x^{n-1}e^{\int_{0}^{x}dv\frac{v}{R(v)}}  \Big|\frac{d f}{dx}(x)\Big|^{2} .
 \end{align}
  The equality~(\ref{IntPart})  extends to all elements in $\mathbf{D}$ and implies that $ \|\mathcal{L}^{(R)}f\|_{2,R}\leq r_{2}\|\Delta_{n} f\|_{2,R}   $ since $R(x)\leq r_{2}$.  Hence,   $\mathcal{L}^{(R)}$ maps $\mathbf{D}$ into $L^{2}\big(\R^{+},w(x)dx\big)$, and $(\mathcal{L}^{(R)},\mathbf{D})$ is relatively bounded to $(\Delta_{n},\mathbf{D})$.

\vspace{.4cm}


\noindent (iii).  Next I  focus on showing that $\Delta_{n}$ is also relatively bounded to $\mathcal{L}^{(R)}$.  Combining~(\ref{IntPart}) with $R(x)\geq r_{1}$ implies that
\begin{align}
 \big\|\mathcal{L}^{(R)}f\big\|_{2,R}^{2}\geq  r_{1}^2 \big\| \Delta_{n}  f \big\|_{2,R}^{2}- \int_{\R^{+}}dx x^{n-1}e^{\int_{0}^{x}dv\frac{v}{R(v)}}  \Big|\frac{df}{dx}(x)\Big|^{2}. \label{IntPart2}
\end{align}

 With the lower bound~(\ref{IntPart2}), it will be enough to demonstrate that there is a $C>0$ such that
\begin{align}\label{Onion}
\int_{\R^{+}}dx x^{n-1}e^{\int_{0}^{x}dv\frac{v}{R(v)}}  \Big|\frac{df}{dx}(x)\Big|^{2}\leq C\|f\|_{2,R}^{2} + \frac{r_{1}^2}{2}\big\| \Delta_{n} f  \big\|_{2,R}^{2} . 
\end{align}
  It is convenient to split the integration over $\R^{+} $ into the domains $x\leq L$ and $x> L$ for some $L\gg 1$ to get the bound 
\begin{align}\label{Jipper}
\int_{\R^{+}}dx x^{n-1} e^{\int_{0}^{x}dv\frac{v}{R(v)}}  \Big|\frac{df}{dx}(x)\Big|^{2} \leq e^{\frac{L^{2}}{2r_{1}}} \int_{\R^{+}}dxx^{n-1}\Big|\frac{df}{dx}(x)\Big|^{2}+\int_{x\geq L}dx x^{n-1} e^{\int_{0}^{x}dv\frac{v}{R(v)}}   \Big|\frac{df}{dx}(x)\Big|^{2}.  
\end{align}

For the first term on the right side of~(\ref{Jipper}), using integration by parts, Cauchy-Schwarz, and  the inequality $2uv\leq u^{2}+v^{2}$ yields the first inequality below for any $c>0$: 
\begin{align}\label{Ghandi}
 \int_{\R^+}dx x^{n-1}\Big|\frac{df}{dx}(x)\Big|^{2}\leq & c \int_{\R^+}dx x^{n-1}\big| f(x)\big|^{2}+\frac{1}{c}\int_{\R^+}dx x^{n-1}\big|(\Delta_{n}f)(x)\big|^{2}\nonumber \\ \leq &  cr_{2}\|f\|_{2,R}^{2}+\frac{r_{2}}{c}\big\|\Delta_{n} f  \big\|_{2,R}^{2}.
\end{align}
The second inequality of~(\ref{Ghandi}) follows from the relation $w(x)\geq r_{2}^{-1}$.   For the second term on the right side of~(\ref{Jipper}), I have the inequalities
 \begin{align}\label{SkyLim}
\int_{x\geq L}dx x^{n-1} e^{\int_{0}^{x}dv\frac{v}{R(v)}}   \Big|\frac{df}{dx}(x)\Big|^{2} &\leq  \frac{r_{2}}{L^{2}}\Big\|   x\frac{d}{d x} f \Big\|_{2,R}^{2} \leq \frac{4r_{2}^{3} }{L^{2}}\big\| \Delta_{n}f \big\|_{2,R}^{2}.
\end{align}
The first inequality in~(\ref{SkyLim}) is Chebyshev's, and the second inequality is discussed below.  By  writing $\mathcal{L}^{(R)}f=x\frac{d}{d x}f+R(x)\Delta_{n}f   $ and expanding the left side of~(\ref{IntPart}), I obtain the following inequality:
\begin{align*}
 \Big\|   x\frac{d}{d x} f \Big\|_{2,R}^{2}\leq &-2\textup{Re}\Big(\Big\langle x\frac{d}{d x} f \Big| R(x)\Delta_{n} f \Big\rangle_{2,R} \Big)\nonumber   \\ \leq &   2r_{2}\Big\|   x\frac{d}{d x} f \Big\|_{2,R} \big\|  \Delta_{n}  f \big\|_{2,R}.
 \end{align*}
 The second inequality is by Cauchy-Schwarz and $R(x)\leq r_{2}$.  Thus $\big\|   x\frac{d}{d x} f \big\|_{2,R}$ is smaller than $2r_{2} \big\| \Delta_{n}  f \big\|_{2,R} $ as required to get the second inequality of~(\ref{SkyLim}).  
 
By picking $L\in \R^{+}$ with $ L^{2}\geq \frac{16r_{2}^{3} }{r_{1}^2 }$ and $c\in \R^{+}$ with $c\geq e^{\frac{L^{2}}{2r_{1}}}\frac{4r_{2}}{r_{1}^{2} }$, I obtain the inequality~(\ref{Onion}) for $C=cr_{2} e^{\frac{L^{2}}{2r_{1}}}  $.

\end{proof}

In the statement of the proposition below, I denote the maximum element in the spectrum of $\mathcal{L}^{(R)}$ by $\overline{\Sigma}\big(\mathcal{L}^{(R)}\big)$.  For $f:\R^{+}\rightarrow \R$, I refer to a point where $\Delta_{n}f$ changes signs as a \textit{radial inflection point}.

\begin{proposition} Let  $R\in B(\R^{+},[r_{1},r_{2}])$  and $f\in \mathbf{D} $.   \label{PropBasics}

\begin{enumerate}

\item  The operator $\mathcal{L}^{(R)}$ has compact resolvent.  

\item The eigenvalues for $\mathcal{L}^{(R)}$ are strictly negative.

\item The principle eigenvalue $\overline{\Sigma}\big(\mathcal{L}^{(R)}\big)$ is non-degenerate, and the phase of the corresponding  eigenfunction can be chosen so that  the following properties  hold: 
\begin{itemize}
\item The values $\phi(x)$ are strictly positive for all $x\in \R^{+}$.  

\item $\Delta_{n}\phi  \in L^{2}\big(\R^{+},w(x)dx\big)$ and $R(x)\big(\Delta_{n}\phi \big)(x)$ is continuous.     

\item The function $\phi$ is strictly decreasing.

\item The function $\phi$ has a unique radial inflection point $\mathbf{c}>0$ at which  $\Delta_{n}\phi$ is continuous (and thus $(\Delta_{n}\phi)(\frak{c})=0$).

\end{itemize}

\item   The following equality holds for any $b\in \R^{+}$:
 $$  \lim_{s\rightarrow \infty}\frac{2\log\Big( \int_{\R^{+}}dx\, \psi_{b,s }^{(R)}(x) f(x) \Big)}{s}= \overline{\Sigma}\big(\mathcal{L}^{(R)}\big)  $$

\end{enumerate}

\end{proposition}

\begin{proof}\text{   }\\
\noindent Part (1): Define the functions $v_{\pm}:\R^{+}\rightarrow \R^+$ such that $v_{-}(x):=1$ and 
$ v_{+}(x):= \int_{x}^{\infty}dz z^{1-n} e^{-\int_{0}^{z}dv\frac{ v  }{R(v)} }    $.
Notice that $ g=v_{\pm} $ are the fundamental solutions to the differential equation 
$$ \big(\mathcal{L}^{(R)}g\big)(x)=  x\frac{dg}{d x}(x)+R(x)\big(\Delta_{n}g\big)(x)=  0.   $$
Also, define the functions $c_{\pm}:\R^{+}\rightarrow \R^{+}$ as
 $$c_{+}(x):= \frac{x^{n-1}}{R(x)}e^{\int_{0}^{x}dv\frac{v}{R(v)}}\hspace{1cm}\text{and}\hspace{1cm}c_{-}(x):= \frac{x^{n-1}}{R(x)}e^{\int_{0}^{x}dv\frac{v}{R(v)}}\int_{x}^{\infty}dzz^{1-n}e^{-\int_{0}^{z}dv\frac{v}{R(v)}}.  $$ 
 By the standard technique of pasting together the fundamental solutions, the Green function $G:\R^{+}\times \R^{+}\rightarrow \R$ satisfying $-\big((\mathcal{L}^{(R)})^{-1}f\big)(x)=\int_{\R^{+} }dzG(x,z)f(z)$ can be written in the form
\begin{align}\label{Haifa}
G(x,z)=c_{-}(z)v_{-}(x)\chi(x\leq z)+c_{+}(z)v_{+}(x)\chi(x>z).
\end{align}
There is a canonical isometry from  $L^{2}\big(\R^{+},w(x)dx\big)$ to $L^{2}(\R^{+})$   given by the map sending $f(x)$ to     $w^{-\frac{1}{2}}(x)f(x)$.   Thus  the kernel 
\begin{align}\label{Spike}
\widehat{G}(x,z):=& w^{\frac{1}{2}}(x)G(x,z)w^{-\frac{1}{2}}(z)\nonumber \\ =&\frac{(xz)^{\frac{n-1}{2}}}{\sqrt{R(x)R(z)}    } e^{\frac{1}{2}\int_{0}^{x}dv\frac{ v  }{R(v)}+\frac{1}{2}\int_{0}^{z}dv\frac{ v  }{R(v)} }  v_{+}\big(\max(x,z) \big)   
\end{align}
yields the Hilbert-Schmidt norm though the standard formula
$
\big\|  (\mathcal{L}^{(R)})^{-1} \big\|_{\textup{HS}}^{2}  = \int_{\R^+\times \R^{+} }dxdz\big|\widehat{G}(x,z)\big|^{2}$.  
However, the quantity $ \int_{\R^+\times \R^{+}}dxdz\big|\widehat{G}(x,z)\big|^{2}$  is finite given the form~(\ref{Spike}).  Since Hilbert-Schmidt operators are compact, the operator $\mathcal{L}^{(R)}$ has compact resolvent.

\vspace{.5cm}

\noindent Part (2):  The largest eigenvalue of $\mathcal{L}^{(R)}$ is the negative  inverse of the largest eigenvalue for $-\big(\mathcal{L}^{(R)}\big)^{-1}$.   Since  $-\big(\mathcal{L}^{(R)}\big)^{-1}$ has a strictly positive integral kernel $G(x,z)$,   the eigenfunction $\phi$ associated with the leading eigenvalue of $-\big(\mathcal{L}^{(R)}\big)^{-1}$ is strictly positive-valued (for the correct choice of phase) and unique.  The leading eigenvalue for $-\big(\mathcal{L}^{(R)}\big)^{-1}$ is  positive and  given by the convex integral of values 
\begin{align}\label{Hab}
\int_{\R^{+}}dz\frac{\phi(z)}{\|\phi\|_{1}}\int_{\R^{+}}dx G(x,z).
\end{align}
Note that  I have the following equality:
\begin{align}\label{Jot}
 \int_{\R^{+} }dx G(x,z)=&\frac{ z^{n-1}  e^{\int_{0}^{x}dv\frac{ v  }{R(v)} }   }{ R(z)  } \int_{z}^{\infty}daa^{2-n}e^{-\int_{0}^{a}dv\frac{v}{R(v)}   }  . 
\end{align}


\vspace{.5cm}

\noindent Part (3):  As remarked in Part (2), the eigenfunction $\phi(x)$ with leading eigenvalue $E:=\overline{\Sigma}\big(\mathcal{L}^{(R)}\big)<0$ must be strictly positive for all  $x\in \R^{+}$.    

By Prop.~\ref{LemSelfAdj} $ \Delta_{n}$ is relatively bounded to $\mathcal{L}^{(R)}$, and thus the eigenfunctions of $\mathcal{L}^{(R)}$ lie in the domain of $\Delta_{n}$.  The continuity of $R(x)\big(\Delta_{n}\phi\big)(x) $ follows from the equality
  \begin{align}\label{Jade} - x\frac{d\phi}{d x}(x)= -E\phi(x)+ R(x)\big(\Delta_{n}\phi\big)(x).    
  \end{align}
since $\phi$ and $\frac{d\phi}{dx}$ are continuous.  Since $R(x)\geq r_{1}$ is bounded away from zero,  $\Delta_{n}\phi$ must be  continuous and equal to zero at any radial inflection point for $\phi$.  In terms of the function $\psi(y):=\phi\big(y^{\frac{1}{2-n}}\big)$, the equation~(\ref{Jade}) can be written as
\begin{align}\label{JadeII}
-(2-n)y\frac{d\psi}{dy}(y)=-E\psi(y)+(2-n)^{2}R\big(y^{\frac{1}{2-n}}\big) y^{\frac{2(1-n)}{2-n}}\frac{d^2 \psi}{dy^2} (y).
\end{align}
Since $(2-n)^{2}y^{\frac{2(1-n)}{2-n}}\frac{d^{2}\psi}{dy^2 }(y)=\big(\Delta_{n}\phi\big)(x)$ for $y=x^{2-n}$, a radial inflection point for $\phi$ occurs at the $\frac{1}{2-n}$ power of an inflection point for $\psi$.  Thus it is sufficient to work with $\psi$.

From~(\ref{JadeII}) and $E<0$, we can see that $\frac{d^{2}\psi}{dy^2}(y)$  is negative in a region around the origin,  $y< \mathbf{c}$,  where $\mathbf{c}>0$ denotes the  inflection point closest to the origin over the interval $(0,\infty)$.  An inflection point for $\psi$ must  exist since $\psi$ is positive, continuously differentiable, and decaying at infinity.   By my remark above, $\frac{d^2\psi}{dy^2}(y)=y^{\frac{2(n-1)}{2-n}}\big(\Delta_{n} \phi \big)(y^{\frac{1}{2-n}} )$ must be zero at inflection points.   Recall that  $\psi$ has a Neumann boundary condition  at zero.    Since $\frac{d\psi}{dy}(0)=0$ and the derivative of $\frac{d\psi}{dy}(y)$ is negative over the interval $(0,\frak{c})$, we must have that $\frac{d\psi}{dy}(y)$ is negative over the interval $(0,\frak{c}]$.  It will suffice for me to show that $\frac{d\psi}{dy}$ and $\frac{d^{2}\psi}{dy^{2}}$ are  nonzero for  $y>\frak{c}$.   Suppose to reach a contradiction that there is some point $u\in (\frak{c},\infty)$ such that either
\begin{align}\label{Lulz}
(\textup{i}).\,\,\frac{d\psi }{d y}(u)=0\hspace{1.5cm} \text{or}\hspace{1.5cm} (\textup{ii}).\,\,  \frac{ d^2 \psi  }{ dy^2  }(u)=0.
\end{align}
  I will let $u$ denote the smallest such value.  Notice that I can not have  both  $\frac{d\psi }{d y}(u)=0$ and  $\frac{d^{2}\psi }{d y^2}(u)=0$ since the term $-E\psi(y)$ in~(\ref{JadeII}) is strictly positive.   For the cases~(\ref{Lulz}), the following reasoning applies:
\begin{enumerate}
\item[(i).]   If $\frac{d\psi }{d y}(u)=0$, then the continuous function $R\big(y^{\frac{1}{2-n}} \big)y^{\frac{2(1-n)}{2-n}}\frac{d^2 \psi}{dy^2}$ must be positive over the interval $[ \mathbf{c}  , u]$.   This, however, contradicts equation~(\ref{Jade}) for $y=u$ since the terms on the right side of~(\ref{JadeII}) are both positive.   
\item [(ii).]  If $  \frac{ d^2 \psi  }{ dy^2}(  u)=0$, then  $ \frac{d\psi}{d y}(y)$ must be negative over the interval $[ \mathbf{c}  , u]$.  A  linear approximation of equation~(\ref{JadeII})  about the point $y=u$ yields that
\begin{align}\label{Beck}
\delta \frac{d\psi}{d y}(u)\Big(\frac{1}{2-n} -\frac{E}{(2-n)^2}\Big)+\mathit{O}\big(\delta^{2}\big)= R\Big((u+\delta )^{\frac{1}{2-n}}\Big)\big|u+\delta\big|^{\frac{2(1-n)}{2-n}}  \frac{d^2 \psi}{dy^2} (u+\delta),\quad \quad |\delta|\ll 1 .  
\end{align}
Since $ \frac{d\psi}{d y}(u)$ and $\frac{E}{2-n}$ are negative,  it follows from~(\ref{Beck}) that  $u$ must be  an  inflection point at which the concavity changes from down to up.  However, by my definitions, $\psi(y)$ is concave up over the interval $(\mathbf{c},u)$, which brings me to a contradiction.  
\end{enumerate}
It follows that $\psi(y)$ is strictly decreasing and has exactly one  inflection point over that interval.

\vspace{.5cm}

\noindent Part (4):  Using the backward representation of the dynamics, I have the equality
$$ \int_{\R^{+}}da \psi_{b,s }^{(R)}(a)f(a)= \big(e^{\frac{s}{2}\mathcal{L}^{(R)}}f \big)(b),$$  
where by assumption $f\in \mathbf{D}$ and thus $f,\Delta_{n}f\in L^{2}\big(\R^+,w(x)dx\big)$.   
 The function $e^{s\mathcal{L}^{(R)}}f$ can be written as 
\begin{align*}
 e^{s\mathcal{L}^{(R)}}f= e^{\frac{s}{2}E}\langle \phi|\,  f\rangle_{R}     \, \phi+e^{\frac{s}{2}\mathcal{L}^{(R)}}g    \hspace{1cm}\text{for}\hspace{1cm}g:= f-\langle \phi|\,  f\rangle_{R}     \, \phi,   
 \end{align*}
where, as before, $\phi$ is the eigenfunction for $\mathcal{L}^{(R)}$ corresponding to the leading eigenvalue $E:=\overline{\Sigma}\big(\mathcal{L}^{(R)}\big)$.  Note that  $g\in \mathbf{D}$   by the assumption $f\in \mathbf{D}$.     Let $E_{1}$ be the largest eigenvalue following $E$. 
  I will show that $e^{\frac{s}{2}\mathcal{L}^{(R)}}g$ decays uniformly  with exponential rate $-E_{1}$ as $s\rightarrow \infty$ over any compact interval $[0,L]$. I have  the following inequalities:
\begin{eqnarray}\label{Turnip}
\big\| e^{\frac{s}{2}\mathcal{L}^{(R)}}g\big\|_{2,R}& \leq & e^{\frac{s}{2}E_{1}}\|g\|_{2,R},  \\  \big\| \Delta_{n} e^{\frac{s}{2}\mathcal{L}^{(R)}}g\big\|_{2,R}& \leq &C e^{\frac{s}{2}E_{1}}\Big(   \|g\|_{2,R}+ \big\|\Delta_{n}g\big\|_{2,R} \Big),\label{Pumpkin}
\end{eqnarray}   
 where the second inequality holds for some $C>0$.   The first inequality in~(\ref{Turnip}) uses that $g$ lies in the orthogonal space to $\phi$.  For the second inequality in~(\ref{Turnip}), recall from Prop.~\ref{LemSelfAdj} that $ \Delta_{n}$ and  $\mathcal{L}^{(R)}$ are mutually relative bounded  so that I have the first and third inequalities below for some constants $c,C>0$:
\begin{align}
\big\| \Delta_{n} e^{\frac{s}{2}\mathcal{L}^{(R)}}g\big\|_{2,R}\leq & c\Big(   \big\|e^{\frac{s}{2}\mathcal{L}^{(R)}}g\big\|_{2,R}+ \big\|\mathcal{L}^{(R)}e^{\frac{s}{2}\mathcal{L}^{(R)}}g\big\|_{2,R} \Big)\nonumber \\  \leq  & ce^{\frac{s}{2}E_{1}}\Big(   \| g\|_{2,R}+ \big\|\mathcal{L}^{(R)}g\big\|_{2,R} \Big)\nonumber \\ \leq & Ce^{\frac{s}{2}E_{1}}\Big(   \| g\|_{2,R}+ \big\|\Delta_{n}g\big\|_{2,R} \Big).
 \end{align}
 The second inequality above follows since $\mathcal{L}^{(R)}$ and $e^{\frac{s}{2}\mathcal{L}^{(R)}}$ commute and $g$, $\mathcal{L}^{(R)}g$ are orthogonal to $\phi$.

Next I use~(\ref{Turnip}) and~(\ref{Pumpkin}) to  bound the supremum of $e^{\frac{s}{2}\mathcal{L}^{(R)}}g$ over a finite interval $[0,L]$.   For $L\geq 1$ there must be a point $x\in [0,L]$ such that the first inequality below holds
\begin{align}\label{Mug}  
 \big|(e^{\frac{s}{2}\mathcal{L}^{(R)}}g)(x)\big|  \leq  \sqrt{ r_{2} } \big\|e^{\frac{s}{2}\mathcal{L}^{(R)}}g\big\|_{2,R} \leq   \sqrt{r_{2} }e^{\frac{s}{2}E_{1}}\|g\|_{2,R}.
\end{align}
  For $x$ satisfying~(\ref{Mug}) the fundamental theorem of calculus  applied to the function $e^{\frac{s}{2}\mathcal{L}^{(R)}}g$ gives the first inequality  below:
\begin{align} 
\sup_{y\in [0,L]}\big|\big(e^{\frac{s}{2}\mathcal{L}^{(R)}}g\big)(y)\big|\leq & \big|\big(e^{\frac{s}{2}\mathcal{L}^{(R)}}g\big)(x)\big|+\int_{0}^{L}dz\Big|\frac{d}{dz}\big(e^{\frac{s}{2}\mathcal{L}^{(R)}}g\big)(z)\Big|
\nonumber \\ \leq & C\sqrt{ r_{2}}e^{\frac{s}{2}E_{1}}\|g\|_{2,R}+\sqrt{Lr_{2} }\Big\|\frac{d}{dz}e^{\frac{s}{2}\mathcal{L}^{(R)}}g\Big\|_{2,R} \nonumber \\ \leq & C\sqrt{r_{2}}e^{\frac{s}{2}E_{1}}\|g\|_{2,R}+\sqrt{L\frac{r_{2}^{3} }{r_{1}}  }\big\| \Delta_{n} e^{\frac{s}{2}\mathcal{L}^{(R)}}g\big\|_{2,R}.\label{Gyro}
\end{align}
The second inequality is by Jensen's inequality and $R(x)\leq r_{2}$.  The last inequality in~(\ref{Gyro}) follows from the relation $\| \frac{d}{dx} e^{\frac{s}{2}\mathcal{L}^{(R)}}g\|_{2,R}\leq  \frac{ r_{2}}{ \sqrt{ r_{1}}} \| \Delta_{n} e^{\frac{s}{2}\mathcal{L}^{(R)}}g\|_{2,R}$, which can be seen from the equality~(\ref{IntPart}).  Finally, the last line of~(\ref{Gyro}) decays on the order $e^{\frac{s}{2}E_{1}}$ by~(\ref{Turnip}) and~(\ref{Pumpkin}).

\end{proof}

\section{The extremal strategies }\label{SecExtremal}

For $\frak{c}>0$, define $\ell_{\frak{c}}(x) :=r_{1}+(r_{2}-r_{1})\chi(x>\frak{c}) $.  These functions correspond to extremal strategies in which the random walker switches between from the lowest possible diffusion rate to the highest at a cut-off value $\frak{c}>0$.      I denote the corresponding generator by $\mathcal{L}_{\frak{c}}:=\mathcal{L}^{(\ell_{\frak{c}})}$. 
 An arbitrary function  $R:\R^{+}\rightarrow [r_{1},r_{2}]$ that is increasing and satisfies $\lim_{x\searrow 0}R(x)=r_{1}$ and $\lim_{x\nearrow \infty}R(x)=r_{2}$, i.e., that determines a `reasonable' strategy for the random walker, can be written as a convex combination of the step functions  $\ell_{\frak{c}}$: 
$$ R(x)=\frac{1}{r_{2}-r_{1}}\int_{0}^{\infty}dR(\frak{c})\ell_{\frak{c}}(x).  $$
  By the linear dependence of  $\mathcal{L}^{(R)}$  on  $R$,  the above convex combination extends to the generators:
$$ \mathcal{L}^{(R)}=\frac{1}{r_{2}-r_{1}}\int_{0}^{\infty}dR  (\frak{c})\mathcal{L}_{\frak{c}}.     $$ 
This  suggests that a generator with maximizing principle eigenvalue should have the form $\mathcal{L}_{\frak{c}}$ for some $\frak{c}>0$, which is the main statement of the following lemma.    The uniqueness of the maximizing $\frak{c}>0$ is established in Lem.~\ref{LemExtAnal}.

\begin{lemma}  \label{LemExtremal}
For any measurable function $R:\R^{+}\rightarrow [r_{1},r_{2}]$, the following inequality holds:
$$ \overline{\Sigma}\big(\mathcal{L}^{(R)}\big)\leq \sup_{\frak{c}\in (0,\infty) }\overline{\Sigma}\big(\mathcal{L}_{\frak{c}}\big) . $$
Moreover, the above supremum is attained as a maximum for a value $\frak{c}>0$ satisfying the following property: The unique radial inflection point over the interval $(0,\infty)$  for the eigenfunction $\phi_{\frak{c}}$ corresponding to the eigenvalue $\overline{\Sigma}\big(\mathcal{L}_{\frak{c}}\big)$ (see  Part (3) of Prop.~\ref{PropBasics}) occurs at the value $\frak{c}$.

\end{lemma}

\begin{proof} Pick some $R\in B\big(\R^+,[r_{1},r_{2}]\big)$.   Let $\phi$ be the eigenfunction corresponding to the principle eigenvalue of $\mathcal{L}^{(R)}$ and $\frak{c}>0$ be the unique radial inflection point of $\phi$.  I will prove the following:
\begin{enumerate}
\item[(i).]   For any $R$ that does not have the special form $R=\ell_{\frak{c}}$, there exists a small perturbation $R'=R+dR$ such that $ \overline{\Sigma}\big(\mathcal{L}^{(R')}\big)>\overline{\Sigma}\big(\mathcal{L}^{(R)}\big)$.

\item[(ii).]   The function $f:(0,\infty)\rightarrow (-\infty,0) $ defined by $f(a):= \overline{\Sigma}\big(\mathcal{L}_{a}\big)$ has a maximum. 

\end{enumerate}

\noindent (\textup{i}).  
  The perturbations that I consider will be of the form  
$$ \mathcal{L}^{(R+hA)}= \mathcal{L}^{(R)}+hA(x)\Delta_{n}   $$
 for $h\ll 1$ and  a well-chosen  bounded  function $A:\R^+\rightarrow \R$. 
 By Prop.~\ref{LemSelfAdj} the operator $\Delta_{n}$ is relatively bounded to $\mathcal{L}^{(R)}$.  It follows that operators of the form $A(x)\Delta_{n}$ are also relatively bounded to $\mathcal{L}^{(R)}$ since $A$ is bounded, and  I can use standard perturbation theory~\cite{Kato} to characterize the  leading eigenvalue of  $\mathcal{L}^{(R+hA)}$ for small $h>0$:
\begin{align}\label{Standard}
 \overline{\Sigma}\big(\mathcal{L}^{(R+hA)}\big)=
\overline{\Sigma}\big(\mathcal{L}^{(R)}\big)+ h\big\langle \phi\big|\,   A(x)\Delta_{n}\phi\big\rangle +\mathit{o}(h).
 \end{align}

I need to show that there is an $A$ such that $\big\langle \phi\big|\,   A(x)\Delta_{n}\phi\big\rangle$ is positive and $R+hA\in  B\big(\R^+,[r_{1},r_{2}]\big)$ for $0<h\ll 1$.  By part (3) of Prop.~\ref{PropBasics}, the eigenfunction corresponding to the principle eigenvalue must satisfy that
 \begin{align}\label{Uruguay}
\big(\Delta_{n}\phi)(x)<0 \quad \text{for} \quad x< \mathbf{c} \hspace{.7cm}\text{and}  \hspace{.7cm}\big(\Delta_{n}\phi\big)(x)>0 \quad \text{for} \quad x>  \mathbf{c}
 \end{align}
 for some  $\mathbf{c}>0$.
Define $A:\R^+\rightarrow \R$ to be of the form
 \begin{align}
A(x):= \left\{  \begin{array}{cc} r_{2}-R(x) & \quad \hspace{.1cm} x> \mathbf{c},    \\  \quad & \quad  \\  r_{1}-R(x)   & \, \quad x\leq \mathbf{c} . \end{array} \right.  
\end{align}
Notice that $R(x)+hA(x)$ maps into the interval $[r_{1},r_{2}]$ for every $h\in [0,1]$. 
Since the values $\phi(x)$ are strictly positive by  Part (3) of Prop.~\ref{PropBasics}, the property~(\ref{Uruguay}) implies that the expression $ \big\langle \phi\big|\,   A(x)\Delta_{n}\phi\big\rangle $  must be strictly positive unless $A(x)=0$.  However, $A(x)=0$ implies that $R=\ell_{\mathbf{c}} $.

\vspace{.4cm}

\noindent (\textup{ii}).  I can extend the definition of $\mathcal{L}_{a}$ to  $a\in \{0,\infty\}$ by setting  $\mathcal{L}_{0}= x\frac{d}{dx}+r_{1}\Delta_{n}$ and $\mathcal{L}_{\infty}= x\frac{d}{dx}+r_{2}\Delta_{n}$.   Note that the principal eigenvalue of $\mathcal{L}_{a}$ is the negative inverse of the operator norm of its compact resolvent: $ \overline{\Sigma}(\mathcal{L}_{a})=-\big\| (\mathcal{L}_{a})^{-1} \big\|_{\infty}^{-1}$.   The continuity of $g(a):=\big\| (\mathcal{L}_{a})^{-1} \big\|_{\infty}$ as a function over $a\in [0,\infty]$ can be established through simple estimates of the Green function of $(\mathcal{L}_{a})^{-1}$,  see~(\ref{Haifa}) with $R(x)=\ell_{c}(x)$, and thus  $f(a):=\overline{\Sigma}\big(\mathcal{L}_{a}\big)$ is continuous.

For  $\phi_{0}(x)=e^{-\frac{x^2}{2r_{1}} }$ and $\phi_{\infty}(x)=e^{-\frac{x^2}{2r_{2}}}$, explicit computations yield that
$$\mathcal{L}_{0}\phi_{0}=-n\phi_{0} \hspace{1cm}\text{and}\hspace{1cm}  \mathcal{L}_{\infty} \phi_{\infty}=-n\phi_{\infty}. $$
Since the functions $\phi_{0}$ and  $\phi_{\infty}$ are positive-valued, they must be the respective eigenvectors corresponding the principle eigenvalues of $ \mathcal{L}_{0}$ and $  \mathcal{L}_{\infty}$, respectively.  It follows that $f(0)=f(\infty)=-n$.   Moreover, $f(a)$ can not have maxima at $a=0,\infty$ by part (\textup{i}), and thus a maximum must occur in $a\in (0,\infty)$.

\end{proof}

By Part (4) of Prop.~\ref{PropBasics}, it is sufficient to focus attention on the extremal generators $\mathcal{L}_{\frak{c}}$.   As before let $I_{\nu}$ and $K_{\nu}$ be modified Bessel functions of the first and second kind, respectively, with index $\nu$; see~\cite{Handbook} for basic properties and estimates involving modified Bessel functions.   Recall that 
$$ I_{\nu}(z)=\big(\frac{z}{2} \big)^{\nu}\sum_{k=0}^{\infty}\frac{  \big( \frac{z^{2}}{4} \big)^{k} }{k!\Gamma(\nu+k+1)  }\hspace{1cm}\text{and}\hspace{1cm} K_{\nu}(z)=\frac{\pi}{2}\frac{  I_{-\nu}(z)-I_{\nu}(z) }{\sin(\nu \pi)  }, $$
where $K_{\nu}(z)$ must be defined as a limit of the above relation when $\nu$ is an integer.  The modified Bessel functions have the following asymptotics for $z\gg 1$: 
$$ I_{\nu}(z)=\frac{  e^{z} }{\sqrt{2\pi z} }\big( 1+\mathit{O}(z^{-1})  \big) \hspace{1cm}\text{and}\hspace{1cm}  K_{\nu}(z) =\sqrt{\frac{\pi}{2z}}e^{-z}\big( 1+\mathit{O}(z^{-1})   \big). $$

Define  $ S^{\pm}_{\nu}:\R^{+}\rightarrow \R^{+}$
$$  S^{+}_{\nu}(z):=   z^{-\nu}K_{\nu}(z)   \hspace{1cm}\text{and}\hspace{1cm}   S^{-}_{\nu}(z):=  z^{-\nu}I_{\nu}(z). $$
\begin{identities}
The identities $I_{\nu}'(z)=I_{\nu\pm 1}(z)\pm\frac{\nu}{z}I_{\nu}(z)$ and  $K_{\nu}'(z)=-K_{\nu\pm 1}(z)\pm\frac{\nu}{z}K_{\nu}(z)$
\begin{enumerate}

\item  $ \frac{d}{dz} \big[ S^{\pm}_{\nu}(z) \big]=\mp z S^{\pm}_{\nu+1}(z)     $

\item $\frac{d}{dz}\big[ S^{\pm}_{\nu}(z)    \big]=\mp S_{\nu-1}^{\pm}(z)-\frac{2\nu}{z}S_{\nu}^{\pm}(z)$

\item  $ \frac{d}{dz} \big[z^{2\nu} S^{\pm}_{\nu}(z) \big]=\mp z^{2\nu-1} S^{\pm}_{\nu-1}(z)     $

\end{enumerate}

\end{identities}

\begin{lemma} \label{LemExtAnal}

Let $r_{1}<r_{2}$ and $V=\sqrt{\frac{r_{2}}{r_{1}} }$.  For $\eta(n,V)$, $\kappa(n,V)$ defined as in Thm.~\ref{ThmMain}, 
$$ \max_{\frak{c}\in \R^{+}  }\overline{\Sigma}\big(\mathcal{L}_{\frak{c}}\big)  =  \eta(n,V)-n.$$
The maximizing value $\frak{c}\in \R^{+}$ is unique and given by $\frak{c}=\kappa(n,V)\sqrt{r_{1} }$.

\end{lemma}

\begin{proof} Let $\phi_{\frak{c}}$ denote  the eigenfunction of $\mathcal{L}_{\frak{c}}$ corresponding to the principle eigenvalue $ E_{\frak{c}}:=\overline{\Sigma}\big(\mathcal{L}_{\frak{c}}\big)$.   Recall from Part (2) of Prop.~\ref{PropBasics} that $E_{\frak{c}}<0$; in fact, the analysis shows that $E_{\frak{c}}\in (-n,0)$.    In parts (i) and (ii) below, I discuss the equations determining the eigenvalue $E_{\frak{c}}$ and the additional criterion determining $\max_{\frak{c}\in \R^{+}}E_{\frak{c}}$, respectively.

\vspace{.3cm}

\noindent (i).   By Part (3) of Prop.~\ref{LemExtremal}, the values $\phi_{\frak{c}}(x)\in \C$  have a single phase for all $x\in \R$ that can be chosen to be positive.  The function $\phi_{\frak{c}}:\R^+\rightarrow \R^{+}$ satisfies the differential equation 
\begin{align}
0=&-E_{\frak{c}}\phi_{\frak{c}}(x)+x\frac{d\phi_{\frak{c}}}{d x}(x)+r_{1}\big(\Delta_{n}\phi_{\frak{c}}\big) (x)  \hspace{1.5cm}  x\leq \frak{c},\label{First}    \\   0=& -E_{\frak{c}}\phi_{\frak{c}}(x)+x\frac{d\phi_{\frak{c}}}{d x}(x)+r_{2}\big(\Delta_{n}\phi_{\frak{c}}\big) (x)   \hspace{1.5cm} x>\frak{c}  . \label{Second}
\end{align}
The fundamental solutions  to the differential equations~(\ref{First}) and~(\ref{Second}) have the form
 $L^{\pm}_{n,   E_{\frak{c} },  r  }$  for $r=r_{1}$ and $r=r_{2}$, respectively, and 
$$L^{\pm}_{n,  E, r }(x):= \int_{0}^{\infty}dy\, y^{E+n-1 }S^{\pm}_{\frac{n-2}{2} }\Big(\frac{xy}{r}\Big)e^{-\frac{x^2+y^2}{2r}}. $$

 Hence the function $\phi_{\frak{c}}$ is a linear combination of  $L^{-}_{n,   E_{\frak{c} },  r_{1}  }$, $L^{+}_{n,   E_{\frak{c} },  r_{1}  }$ over the domain $x\leq \frak{c}$ and a linear combination of $L^{+}_{ r_{2}, E_{\frak{c}}}$, $L^{-}_{ r_{2}, E_{\frak{c}}} $ over the domain $x> \frak{c}$.  In order for the function $\phi_{\frak{c}}$ to be positive, be an element of $\mathbf{D}  $, and have the  boundary condition $\lim_{x\searrow 0}x^{n-1}\frac{d\phi }{dx}(x)=0  $, it must have the following unnormalized form:  
\begin{align}\label{Form}
\phi_{\frak{c}}(x)= \left\{  \begin{array}{cc} L^{-}_{n,   E_{\frak{c} },  r_{1}  }(x)& \quad \hspace{.1cm} x\leq  \frak{c},    \\ \gamma L^{+}_{n,   E_{\frak{c} },  r_{2}  }(x)   &  \quad  x> \frak{c}, \end{array} \right.  
\end{align}
for some constant $\gamma\in \R^{+}$.   The values $\gamma$ and $E_{\frak{c}}$ are fixed by the requirement that $\phi_{\frak{c}}$ is continuously differentiable at $x=\frak{c}$.  Equivalently, $E_{\frak{c}}$ can be determined first through the Wronskian identity $W\big(L^{-}_{n,   E_{\frak{c} },  r_{1}  }, L^{+}_{n,   E_{\frak{c} },  r_{2}  }\big)(\frak{c})=0$, and then $\gamma$ is given by
$
\gamma= \frac{ L^{-}_{n,   E_{\frak{c} },  r_{1}  }(\frak{c})}{   L^{+}_{n,   E_{\frak{c} },  r_{1}  }(\frak{c})  }$. 

To see that the equation $W\big(L^{-}_{n,   E_{\frak{c} },  r_{1}  }, L^{+}_{n,   E_{\frak{c} },  r_{2}  }\big)(\frak{c})=0$ has a solution for some $E_{\frak{c}}\in(-n,0)$, notice that the Wronskian equaling zero is equivalent to 
$$H_{r_{1},r_{2}}^{(\frak{c})}	(E_{\frak{c}})=1 \hspace{1cm} \text{for} \hspace{1cm} H_{r_{1},r_{2}}^{(\frak{c})}(E):= \frac{  \frac{d L^{-}_{n,   E_{\frak{c} },  r_{1}  }  }{d x }(\frak{c})    L^{+}_{n,   E_{\frak{c} },  r_{2}  }(\frak{c})      }{ \frac{d L^{+}_{n,   E_{\frak{c} },  r_{2}  } }{d x } (\frak{c}) L^{-}_{n,   E_{\frak{c} },  r_{1}  } (\frak{c})      }  ,  $$
and notice that for any fixed $\frak{c}$, $r_{1}$, and $r_{2}$
$$ \lim_{ E\nearrow 0  }H_{r_{1},r_{2}}^{(\frak{c})}(E)= 0 \hspace{1cm} \text{and} \hspace{1cm} \lim_{ E\searrow -n }H_{r_{1},r_{2}}^{(\frak{c})}(E)= \frac{r_{2}}{r_{1}}>1.  $$
The intermediate value theorem guarantees that there exists a solution $H_{r_{1},r_{2}}^{(\frak{c})}	(E_{\frak{c}})=1$  for some $E_{\frak{c}}\in(-n,0)$ .   The solution $E_{\frak{c}} $ must be unique since otherwise it would be possible to construct two positive-valued eigenfunctions $\phi_{\frak{c},1}(x)$ and $\phi_{\frak{c},2}(x)$ for $\mathcal{L}_{\frak{c}}$ of the form~(\ref{Form}).   However,   $(\mathcal{L}^{(R)},\mathbf{D})$ is self-adjoint in the weighted Hilbert space $L^{2}(\R^{+}, w(x)dx)$  by Prop.~\ref{LemSelfAdj} so   $\phi_{\frak{c},1}(x)$ and $\phi_{\frak{c},2}(x)$ must be orthogonal, which contradicts the possibility of both functions being strictly positive.

Note that the continuous differentiability of $\phi_{\frak{c}} $ along with the equations~(\ref{First}) and~(\ref{Second}) imply that $\Delta_{n}\phi_{\frak{c}}$ is discontinuous at $x=\frak{c}$ unless $\lim_{x\rightarrow \frak{c}   } \big(\Delta_{n}\phi_{\frak{c}}\big) (x) =\big(\Delta_{n}\phi_{\frak{c}}\big) (\frak{c}) =0  $.

\vspace{.4cm}

\noindent (ii). By Lem.~\ref{LemExtremal} the parameter value $\frak{c}\in \R^{+}$ at which $E_{\frak{c}}$ is maximized also has the property that  the eigenfunction $\phi_{\frak{c}}$ has a radial inflection point at $\frak{c}$.   Recall from part (3) of Prop.~\ref{PropBasics} that the radial  inflection point of $\phi_{\frak{c}}$ must be a continuity point for $\Delta_{n} \phi_{\frak{c}}$:    $ \lim_{x\rightarrow \frak{c}   }\big(\Delta_{n} \phi_{\frak{c}}\big)  (x)=\big(\Delta_{n} \phi_{\frak{c}}\big)  (\frak{c})=0 $.  These results can be alternatively found by  combining the relations $W\big(L^{+}_{n, E_{\frak{c}}, r_{1  }}   , L^{-}_{n, E_{\frak{c}},r_{2  }}\big)(\frak{c})=0$ and $\frac{dE_{\frak{c}}}{d\frak{c}}=0$.
  The constraint  $\lim_{x\rightarrow \frak{c}   }\big(\Delta_{n}\phi_{\frak{c}}\big) (x)=0  $  and the form (\ref{Form}) yield that
 \begin{align}\label{Ribbit}
 0= \big(\Delta_{n} L^{+}_{n, E_{\frak{c}}, r_{2} }\big) (\frak{c})\hspace{1cm}\text{and}\hspace{1cm} 0= \big(\Delta_{n} L^{-}_{n, E_{\frak{c}}, r_{1} }\big) (\frak{c})  . 
  \end{align}
Moreover, combining equations~(\ref{First}) and (\ref{Second}) with~(\ref{Ribbit})  implies that  the values $\frak{c}$, $E_{\frak{c}}$ satisfy the equations
$$
 \frac{E_{\frak{c}}}{\frak{c}}L^{+}_{n, E_{\frak{c}}, r_{2} }(\frak{c})  =\frac{dL^{+}_{n, E_{\frak{c}}, r_{1  }} }{d x}(\frak{c}) \hspace{1cm}\text{and}\hspace{1cm} \frac{E_{\frak{c}}}{\frak{c}}L^{-}_{n, E_{\frak{c}}, r_{1} }(\frak{c})  =\frac{dL^{-}_{n, E_{\frak{c}}, r_{1  }} }{d x}(\frak{c}) .
$$

Denote $\kappa:=\frac{\frak{c}}{\sqrt{ r_{1}} } $, $V:= \sqrt{\frac{r_{2}}{r_{1}} } $,   $\eta:=E_{\frak{c}}+n$, $\nu=\frac{n-2}{2}$,  and $Y_{\nu, \eta }^{\pm}(\kappa):= \int_{0}^{\infty}dz  z^{\eta-1} S^{\pm}_{\nu }(z\kappa)  e^{-\frac{z^2+\kappa^2}{2}} $.    By changing  variables $\frac{y}{\sqrt{r_{1}} }\rightarrow y$ in the integrals defining $L^{-}_{n, E_{\frak{c}}, r_{1} }(\frak{c})$ and  $L^{+}_{n, E_{\frak{c}}, r_{2} }(\frak{c})$,  the above equations are equivalent to 
\begin{align}\label{NEta}
 \frac{n-\eta    }{ \frac{ \kappa  }{V}}Y_{\nu,\eta }^{+}\Big(\frac{\kappa}{V}\Big) = -\frac{dY_{\nu,\eta }^{+}}{dx}\Big(\frac{\kappa}{V}\Big) \hspace{1cm}\text{and}\hspace{1cm}  \frac{n-\eta  }{ \kappa } Y_{\nu ,\eta }^{-}(\kappa) =   -\frac{dY_{\nu,\eta }^{-}}{dx}(\kappa)  .
\end{align}
The left side of~(\ref{NEta}) is equivalent to the left side of~(\ref{Krakow}) by the identity 
$$\frac{dY_{\nu, \eta }^{-}}{dx}(x)=-x Y_{\nu, \eta }^{-}(x)+x Y_{\nu+1, \eta+2 }^{-}(x)$$     
 for $x=\frac{\kappa}{V}$.  The right side of~(\ref{NEta}) is equivalent to the right side of~(\ref{Krakow}) by using the identity above and rewriting $xY_{\nu+1, \eta+2 }^{-}(x)$ with the following:
 \begin{align*}
 x Y_{\nu+2, \eta+1 }^{-}(x)=& x \int_{0}^{\infty}dz  z^{\eta+1} S^{-}_{\nu+1 }(zx)   e^{-\frac{z^2+x^2}{2}}     \\  =&  x^{2}  \int_{0}^{\infty}dz  z^{\eta} \frac{dS^{-}_{\nu+1 }}{dx}(zx)   e^{-\frac{z^2+x^2}{2}}+x  \eta  \int_{0}^{\infty}dz z^{\eta-1} S^{-}_{\nu+1 }(zx)   e^{-\frac{z^2+x^2}{2}}   \\  =& x Y_{\nu, \eta }^{-}(x)+ (\eta-n) x \int_{0}^{\infty}dz  z^{\eta-1} S^{-}_{\nu +1 }(zx)   e^{-\frac{z^2+x^2}{2}}  \\   =&x Y_{\nu,\eta }^{-}(x)+ (\eta-n)x^{-2\nu+1}  \int_{0}^{\infty}dz z^{\eta-1}\Big(  z^{-2\nu+1} \int_{0}^{ z\kappa  }da a^{2\nu+1} S^{-}_{\nu }(a)  \Big) e^{-\frac{z^2+x^2}{2}}
\\   =&x Y_{\nu, \eta }^{-}(x)+ (\eta-n)\kappa^{-2\nu+1}  \int_{0}^{\infty}dz z^{\eta-1}\Big(   \int_{0}^{ x  }da a^{2\nu+1} S^{-}_{\nu }(za)  \Big) e^{-\frac{z^2+x^2}{2}}
\\   =&x Y_{\nu, \eta }^{-}(x)+ (\eta-n)x^{-2\nu+1} e^{-\frac{x^2}{2}}\int_{0}^{ x  }da\Big( \int_{0}^{\infty}dz z^{\eta-1}    S^{-}_{\nu }(za)  e^{-\frac{z^2+a^2}{2}}\Big) a^{2\nu+1} e^{\frac{a^2}{2}}
\\   =&x Y_{\nu, \eta }^{-}(x)+ (\eta-n)\kappa^{1-n} e^{-\frac{x^2}{2}}\int_{0}^{x}da   Y_{\nu, \eta }^{-}(a)  a^{n-1}e^{\frac{a^2}{2}},
\end{align*}
where the second equality applies integration by parts.   The third equality applies the identity $ \frac{dS^{-}_{\nu+1 }}{dx}(x)=\frac{1}{x}S^{-}_{\nu }(x) -\frac{ 2\nu+2  }{ x }S^{-}_{\nu+1 }(x) $ and $n=2\nu+2$.  The fourth applies the fundamental theorem of calculus and the identity $\frac{d}{dx}[ x^{2\nu+2}  S^{-}_{\nu+1 }(x) ]=  x^{2\nu+1}  S^{-}_{\nu }(x) $.  The fifth equality changes variables in the inner integration, and  the  sixth swaps the order of integration and rearranges terms.  The last equality uses the definitions of $\nu$ and $Y_{\nu, \eta }^{-}$.

\end{proof}

\section{Proof of Theorem~\ref{ThmMain}}

Now I am almost ready to move to the  proof of Thm.~\ref{ThmMain}, which requires finding the maximum over all $R\in B_{r_{1},r_{2}}$ of the quantity
\begin{align*}
\lim_{\epsilon\searrow 0}\left\{ n-\frac{\log\Big( \int_{ [0,   \epsilon ] }da \mathcal{P}^{(R)}_{y,T}(a) \Big)}{  \log(\epsilon) }\right\},
\end{align*}
where the integral in the $\log$ can be written in terms of the time-changed Bessel process $\mathbf{x}_{t}$ or the stationary process $Z_{s}$ as
\begin{align}\label{Gim}
\int_{[0,  \epsilon]}da \mathcal{P}^{(R)}_{y,T}(a) =  \mathbb{P}_{y}\big[  \mathbf{x}_{T}\leq \epsilon   \big]= \mathbb{E}_{b}\Big[\chi\Big( \lim_{s\rightarrow \infty}  e^{-\frac{s}{2} }  Z_{s} \leq \frac{\epsilon}{\sqrt{T}}   \Big)\Big]  
\end{align}
for $b=\frac{y}{\sqrt{T}}$.  The limit in the right  expectation  exists by the remark~(\ref{LimitEX}).  
  My strategy is to use previous results about the backwards generator $\mathcal{L}^{(R)}$ of the process $Z_{s}$, however, I have no results that are directly applicable to  the study of expectations involving the random variable $ \lim_{s\rightarrow \infty}e^{-s} Z_{s}$.  It will be convenient to use the Markov property to rewrite~(\ref{Gim})  as in the first equality below: 
\begin{align}\label{Tip}
 \int_{[0,\epsilon]}da\mathcal{P}_{y,T}^{(R)}(a)=&\int_{\R^+}da\mathcal{P}_{y,\, T- \epsilon^2}^{(R)}(a)F_{\epsilon,\, \epsilon^2}^{(R)}(a)\nonumber      \\
=&\int_{\R^+}da\psi_{y,\,s }^{(R)}(a)F_{1,1}^{(R)}(a),
\end{align}
where $s=\log\big(\frac{T}{\epsilon^2}\big)$ and  $F_{\epsilon, t}^{(R)}:\R^+\rightarrow \R^+$ is defined as 
\begin{align*}
F_{\epsilon, t}^{(R)}(a):=\mathbb{P}\big[\mathbf{x}_{T}\leq \epsilon\,\big|\,\mathbf{x}_{T-t}=a   \big]=\mathbb{E}_{\frac{a}{\sqrt{t} } } \Big[\chi\Big(\lim_{s\rightarrow \infty} e^{-\frac{s}{2}}Z_{s}\leq \frac{\epsilon}{\sqrt{t}}    \Big)  \Big]. 
\end{align*}
The second equality in~(\ref{Tip}) uses that the function $F_{\epsilon, t}^{(R)}$ has the convenient scale-invariance
\begin{align*}
\hspace{2.2cm} F_{\epsilon,t}^{(R)}(x)=F_{\lambda\epsilon,\lambda^2 t}^{(R)}(\lambda x)  \hspace{1cm}\text{for } \lambda>0.
\end{align*}
For the last equality in~(\ref{Tip}), I have changed integration variables and used that $ \psi_{y,\,s }^{(R)}(a):=\epsilon \mathcal{P}_{y,\, T-\epsilon^2}^{(R)}(\epsilon a   )$ when $\epsilon=\sqrt{T}e^{-\frac{s}{2}}$.  The bottom expression for $ \int_{[0,\epsilon]}da\mathcal{P}_{y,T}^{(R)}(a)$ in~(\ref{Tip})  has a promising similarity to the expressions appearing in part (4) of Prop.~\ref{PropBasics};   Lemma~\ref{H} below verifies the conditions of  Prop.~\ref{PropBasics}.

Notice that $F_{\epsilon, t}^{(R)}$ has the form 
$G_{h}^{(R)}(x):=\mathbb{E}_{x}\big[h\big(\lim_{s\rightarrow \infty}e^{-\frac{s}{2}}Z_{s}  \big) \big]$
for $h(a)=\chi\big(a\in [0,\frac{\epsilon}{\sqrt{t}} ] \big)$.    The following lemma shows that $G_{h}^{(R)}$ lies in the domain of  $\mathcal{L}^{(R)}$ when  $h$ is compact and smooth.   This will  be useful for the proof of Thm.~\ref{ThmMain} since  $F_{\epsilon, t}^{(R)}$ can be bounded  above and below by functions of the form  $G_{h}^{(R)}$ for  $h$ compact and smooth.

\begin{lemma}\label{H}
Let $h\in \mathbf{D}$ be smooth and have compact support and $R\in   B_{r_{1},r_{2} }$.  The function $G_{h}^{(R)}$ is an element of $\mathbf{D}$.  

\end{lemma}

\begin{proof} 
I can assume without losing generality that $h(x)$ is nonnegative.  To verify that $G_{h}^{(R)}\in \mathbf{D}$, I must show the following:
$$ (\textup{i}).\, \, G_{h}^{(R)} \in L^2(\R^+,w(x)dx)   \hspace{1cm} (\textup{ii}). \,\, \Delta_{n} G_{h}^{(R)} \in  L^2(\R^+,w(x)dx) \hspace{1cm}  (\textup{iii}).\,\,  \lim_{x\searrow 0}x^{n-1}\frac{ dG_{h}^{(R)} }{dx}(x) =0  $$

\noindent (i). Since $h$ has compact support, it is smaller than $c1_{[0,L]}$ for some $c,L>0$.    Let $\tau\in [0,\infty]$ be the hitting time that $e^{-\frac{s}{2}Z_{s}}$ reaches a value in $[0,2L]$.  I have the inequality below:
\begin{align}\label{GBound}
G_{h}^{(R)}(x):=\mathbb{E}_{x}\Big[h\Big(\lim_{s\rightarrow \infty}e^{-\frac{s}{2}}Z_{s}    \Big)\Big]\leq c\mathbb{P}_{x}\big[\tau<\infty].
\end{align}
In particular $G_{h}^{(R)}(x)$ is uniformly bounded by $c$, and the problem of showing that $ G_{h}^{(R)}\in L^2(\R^+,w(x)dx) $ reduces to showing that $\mathbb{P}_{x}\big[\tau<\infty]$ has sufficient decay as $x$ goes to infinity.  

The random variable  $e^{-\frac{s}{2}}Z_{s}$ can be written as the following stochastic integral:
\begin{align}\label{Expon}
e^{-\frac{s}{2}}Z_{s}=x+\int_{0}^{s}e^{-\frac{r}{2}}\Big( \frac{n-1}{Z_{r}}dr+\sqrt{R(Z_{r})}d\mathbf{B}_{r}'    \Big).
\end{align}
Assume $x>2L$ and let $\delta$ be picked from the interval $(0,1)$.  The integral equation~(\ref{Expon}) yields the equality below:
\begin{align}\label{Tywin}
\mathbb{P}_{x}\big[\tau<\infty]=&\mathbb{P}_{x}\bigg[\bigg| x+\int_{0}^{\tau}e^{-\frac{r}{2}}\Big(\frac{n-1}{2Z_{r}}dr+\sqrt{R(Z_{r}) }d\mathbf{B}_{r}'      \Big) \bigg|\leq 2L \bigg]\nonumber  \\
\leq &\mathbb{P}_{x}\bigg[\bigg| \int_{0}^{\tau}e^{-\frac{r}{2}}\Big(\frac{n-1}{2Z_{r}}dr+\sqrt{R(Z_{r}) }d\mathbf{B}_{r}'      \Big) \bigg|\leq x-2L \bigg] \nonumber
\\
\leq &\mathbb{P}_{x}\bigg[\bigg| \int_{0}^{\tau}e^{-\frac{r}{2}}\sqrt{R(Z_{r}) }d\mathbf{B}_{r}'      \bigg|\leq x- \widehat{L} \bigg], \nonumber
\intertext{where $\widehat{L}:=2L+\frac{|n-1|}{2L}$, and the second inequality uses that $Z_{r}\leq 2L$ for $r\in [0,\tau]$.  By Chebyshev's inequality, the above is smaller than }
\leq & e^{-\frac{1-\delta}{2r_{2}}(x-\widehat{L})^2 } \mathbb{E}_{x}\bigg[  e^{\frac{1-\delta}{2r_{2}}\big(  \int_{0}^{\tau}e^{-\frac{r}{2}}\sqrt{R(Z_{r}) }d\mathbf{B}_{r}'  \big)^2       }  \bigg]     \nonumber \\ 
\leq & \delta^{-\frac{1}{2}} e^{-\frac{1-\delta}{2r_{2}}(x-\widehat{L})^2 } .
\end{align}
The second inequality above holds because $f(x)=e^{\frac{1-\delta}{2r_{2}}x^2    }$ is a convex function and $R(z)\leq r_{2}$.   Thus the expectation in the line above~(\ref{Tywin}) is larger when $R(Z_{r})$ is replaced by $r_{2}$ and $\tau$ is replaced by $\infty$:
$$  \mathbb{E}_{x}\bigg[  e^{\frac{1-\delta}{2r_{2}}\big(  \int_{0}^{\tau}e^{-\frac{r}{2}}\sqrt{R(Z_{r}) }d\mathbf{B}_{r}'  \big)^2       }  \bigg]  \leq    \mathbb{E}_{x}\bigg[  e^{\frac{1-\delta}{2}\big(  \int_{0}^{\tau}e^{-\frac{r}{2}} d\mathbf{B}_{r}'  \big)^2       }  \bigg]        =\delta^{-\frac{1}{2} }. $$
The equality uses that $\int_{0}^{\infty}e^{-\frac{r}{2}}d\mathbf{B}_{r}'$ is a mean-zero Gaussian with variance one.  

Putting~(\ref{GBound}) and~(\ref{Tywin}) together, I have that
$$   \big\| G_{h}^{(R)}   \big\|^2_{2,R}\leq \frac{ c\delta^{-\frac{1}{2}}  }{r_{1}}\int_{\R^+}dxx^{n-1}e^{\int_{0}^{x}dy\frac{y}{R(y)} }  e^{-\frac{1-\delta}{2r_{2}}(x-L)^2 },         $$
and hence I can pick $\delta>0$ small enough so that the integral is finite since $R(y)\geq r_{1}$ and $R(y)\nearrow r_{2}$ and $y\rightarrow \infty$.

\vspace{.5cm}

\noindent (ii).  Now I will show that $\Delta_{n}G_{h}^{(R)}\in L^{2}(\R^{+},w(x)dx)$.  Notice that by inserting a conditional expectation into the formula for $G_{h}^{(R)}(x)$ we can write
\begin{align}\label{Dun}
G_{h}^{(R)}(x)=\mathbb{E}_{x}\Big[h\Big( \lim_{s\rightarrow \infty} e^{-\frac{s}{2}}Z_{s}  \Big) \Big]=&\mathbb{E}_{x}\Big[\mathbb{E}\Big[ h\Big( e^{-\frac{r}{2}} \lim_{s\rightarrow \infty} e^{-\frac{s-r}{2}}Z_{s}   \Big)  \Big| Z_{r}\Big]  \Big],   \nonumber
\intertext{and a first-order expansion for $0<r \ll 1$ leads to the order equalities  }
=&\mathbb{E}_{x}\Big[G_{h}^{(R)}(Z_{r})  -\frac{r}{2}G_{\widehat{h}}^{(R)}(Z_{r})\Big]+\mathit{O}\big(r^2\big) \nonumber \\
=&G_{h}^{(R)}(x)+\frac{ r}{2}\Big(\big( \mathcal{L}^{(R)} G_{h}^{(R)}  \big) (x)- G_{\widehat{h}}^{(R)}(x)  \Big)+ \mathit{O}\big(r^2\big),
\end{align}
where $\widehat{h}:=x\frac{dh}{dx} $.  The third equality above uses that $\frac{1}{2}\mathcal{L}^{(R)}$ generates the backward dynamics for $Z_{r}$ and 
\begin{align*}
\mathbb{E}\Big[ h\Big( e^{-\frac{r}{2}} \lim_{s\rightarrow \infty} e^{-\frac{s-r}{2}}Z_{s}   \Big)  \Big| Z_{r}  \Big]=&\mathbb{E}\Big[ h\Big( \lim_{s\rightarrow \infty} e^{-\frac{s-r}{2}}Z_{s}   \Big)  \Big| Z_{r}  \Big]-\frac{r}{2}\mathbb{E}\Big[\widehat{h}\Big( \lim_{s\rightarrow \infty} e^{-\frac{s-r}{2}}Z_{s}   \Big)  \Big| Z_{r}  \Big]+\mathit{O}\big(r^2\big) 
\\ =&G_{h}^{(R)}(Z_{r})  -\frac{r}{2}G_{\widehat{h}}^{(R)}(Z_{r})+\mathit{O}\big(r^2\big).
\end{align*}

  It follows from~(\ref{Dun}) that    
\begin{align}\label{Till}
\big( \mathcal{L}^{(R)} G_{h}^{(R)}\big)(x)=G_{\widehat{h}}^{(R)}(x).
\end{align}
The function $\widehat{h}(x)$ is smooth and has compact support by my assumptions on $h(x)$.  By the analysis above, 
$G_{h}^{(R)}$ is an element of $L^2(\R^+,w(x)dx)$.  It follows from~(\ref{Till})  that $G_{h}^{(R)}$ is in the domain of $\mathcal{L}^{(R)}$.   Finally, $\Delta_{n}G_{h}^{(R)}$ is in $L^2(\R^+,w(x)dx)$ since the operator $\Delta_{n}$ is relatively bounded to $\mathcal{L}^{(R)}$ by Prop.~\ref{LemSelfAdj}.

\vspace{.5cm}

\noindent (iii).  By Lem.~\ref{Nill}, it is sufficient to show that there are $h_{s}\in \mathbf{D}$, $s\in \R^+$ such that  as $s\rightarrow\infty$
\begin{align}\label{Jin}
 \big\|\Delta_{n}\big( h_{s}- G_{h}^{(R)}\big)\big\|_{2,R}\longrightarrow 0.
\end{align}
   The functions $h_{s}(x):=\mathbb{E}_{x}\big[h\big(e^{-\frac{s}{2}}Z_{s}    \big)\big]$ are elements in $\mathbf{D}$ since $h_{s}=e^{\frac{s}{2}\mathcal{L}^{(R)} }g_{s}$ for  $g_{s}(x):= h\big(e^{-\frac{s}{2}}x\big)\in \mathbf{D}$.  The convergence~(\ref{Jin}) can be shown by using similar arguments as above.

\end{proof}

\vspace{.2cm}

The proof  of Thm.~\ref{ThmMain} continues from the basic setup discussed at the beginning of this section

\begin{proof}[Proof of Theorem.~\ref{ThmMain}]
Let $b=\frac{y}{\sqrt{T}}$ and $s=-2\log(\epsilon)$.  By the observation~(\ref{Tip}), I have the first equality 
$\int_{[0,\epsilon]}da\mathcal{P}_{y,T}^{(R)}(a)=\int_{\R^+} dz\psi_{b,s}^{(R)}(a)F_{1,  1}^{(R)}(a) $.
Since $F_{1,1}^{(R)}(z)=G^{(R)}_{1_{[0,1]}}(z)$, the function $F_{1,1}^{(R)}(z)$ can be bounded above and below by $G^{(R)}_{h}(z)$ for $h\in \mathbf{D}$ compact and smooth.
Moreover, by Lem.~\ref{H} I have that $G^{(R)}_{h}\in \mathbf{D}$ for such $h$.  By part (4) of Prop.~\ref{PropBasics}, I have the first equality below:
$$n+\widetilde{\Sigma}\big(\mathcal{L}^{(R)}\big)=\lim_{s\rightarrow \infty}\Bigg\{ n+\frac{2\log\Big(\int_{\R^+}dz\psi_{b,s}^{(R)}(z)G_{h}^{(R)}(z)     \Big)}{  s}    \Bigg\} .  $$
It follows from the above that $n-\frac{\log\big(\int_{[0,\epsilon]}dx\mathcal{P}_{y,T}^{(R)}(x)    \big)}{\log(\epsilon)   }$ also converges to $ n+\overline{\Sigma}\big(\mathcal{L}^{(R)}\big)   $ as $\epsilon\searrow 0$.  By Lems.~\ref{LemExtremal} and~\ref{LemExtAnal}, the value $ n+\widehat{\Sigma}\big(\mathcal{L}^{(R)}\big)  $ attains the maximum $\eta(n,V)$ over all $R\in B_{r_{1},r_{2}}$ when $R$ has the form $R(x)=r_{1}\chi\big( \frac{x}{\sqrt{r_{1}}}\leq \kappa(n,V)\big)+r_{2}\chi\big( \frac{x}{\sqrt{r_{1}}} > \kappa(n,V)\big)     $.

\end{proof}

\section*{Acknowledgments}
  This work is supported by the European Research Council grant No. 227772  and  NSF grant DMS-08446325.

\end{document}